\documentclass[12pt,british,reqno]{amsart}

\usepackage{babel}

\usepackage{amsthm,amssymb,amsmath} 
\usepackage[latin1]{inputenc}
\usepackage{enumerate}

\makeatletter
\def\@map#1#2[#3]{\mbox{$#1 \colon\thinspace #2 \longrightarrow #3$}}
\def\map#1#2{\@ifnextchar [{\@map{#1}{#2}}{\@map{#1}{#2}[#2]}}
\g@addto@macro\@floatboxreset\centering
\makeatother

\renewcommand{\epsilon}{\ensuremath{\varepsilon}}
\renewcommand{\phi}{\ensuremath{\varphi}}

\renewcommand{\to}{\ensuremath{\longrightarrow}}
\renewcommand{\mapsto}{\ensuremath{\longmapsto}}
\newcommand{\R}{\ensuremath{\mathbb R}}
\newcommand{\N}{\ensuremath{\mathbb N}}
\newcommand{\Z}{\ensuremath{\mathbb{Z}}}

\newcommand{\St}[1][2]{\ensuremath{\mathbb S}^{#1}}

\newcommand{\rp}{\ensuremath{\mathbb{R}P^2}}

\newcommand{\brak}[1]{\ensuremath{\left\{ #1 \right\}}}
\newcommand{\ang}[1]{\ensuremath{\left\langle #1\right\rangle}}

\newtheoremstyle{theoremm}{}{}{\itshape}{}{\scshape}{.}{ }{}
\theoremstyle{theoremm}
\newtheorem{thm}{Theorem}
\newtheorem{lem}[thm]{Lemma}
\newtheorem{prop}[thm]{Proposition}
\newtheorem{cor}[thm]{Corollary}
\newtheoremstyle{remark}{}{}{}{}{\scshape}{.}{ }{}

\theoremstyle{remark}

\newtheorem{rem}[thm]{Remark}

\renewcommand{\ker}[1]{\ensuremath{\operatorname{\text{Ker}}\left({#1}\right)}}
\newcommand{\kernb}{\ensuremath{\operatorname{\text{Ker}}}}

\newcommand{\reth}[1]{Theorem~\protect\ref{th:#1}}
\newcommand{\relem}[1]{Lemma~\protect\ref{lem:#1}}
\newcommand{\repr}[1]{Proposition~\protect\ref{prop:#1}}
\newcommand{\reco}[1]{Corollary~\protect\ref{cor:#1}}
\newcommand{\resec}[1]{Section~\protect\ref{sec:#1}}

\newcommand{\req}[1]{equation~(\protect\ref{eq:#1})}
\newcommand{\reqref}[1]{(\protect\ref{eq:#1})}

\numberwithin{equation}{section}
\allowdisplaybreaks

\usepackage[pdftex,%
bookmarks=true,
breaklinks=true,
bookmarksnumbered = true, 
colorlinks= true,
urlcolor= green,
anchorcolor = yellow,
citecolor=blue,
]{hyperref}

\begin{document}

\title{Crystallographic groups and flat manifolds from surface braid groups} 
\author[D.~L.~Gon\c{c}alves]{Daciberg Lima Gon\c{c}alves}
\address{Departamento de Matem\'atica - IME-USP, Rua~do~Mat\~ao~1010,
CEP:~05508-090 - S\~ao Paulo - SP - Brazil}
\email{dlgoncal@ime.usp.br}

\author[J.~Guaschi]{John Guaschi}
\address{Normandie Univ., UNICAEN, CNRS, Laboratoire de Math\'ematiques Nicolas Oresme UMR CNRS 6139, 14000 Caen, France}
\email{john.guaschi@unicaen.fr}

\author[O.~Ocampo]{Oscar Ocampo}
\address{Universidade Federal da Bahia, Departamento de Matem\'atica - IME,
Av.~Adhemar de Barros~S/N, CEP:~40170-110 - Salvador - BA - Brazil}
\email{oscaro@ufba.br}

\author[C.~M.~Pereiro]{Carolina de Miranda e Pereiro}
\address{Universidade Federal do Esp\'{i}rito Santo, UFES, Departamento de Matem\'{a}tica, 29075-910, Vit\'{o}ria, Esp\'{i}rito Santo, Brazil}
\email{carolinapereiro@gmail.com}

\subjclass[2010]{Primary: 20F36, 20H15; Secondary: 57N16.}
\date{27th May 2020}

\keywords{Surface braid groups, crystallographic group, flat manifold, Anosov diffeomorphism, K\"ahler manifold}

\begin{abstract}
\noindent
Let $M$ be a compact surface without boundary,
 and $n\geq 2$. We analyse the quotient group $B_n(M)/\Gamma_2(P_n(M))$ of the surface braid group $B_{n}(M)$ by the commutator subgroup $\Gamma_2(P_n(M))$ of the pure braid group $P_{n}(M)$. If $M$ is different from the $2$-sphere $\St$, we prove that $B_n(M)/\Gamma_2(P_n(M)) \cong P_n(M)/\Gamma_2(P_n(M)) \rtimes_{\varphi} S_n$, and that $B_n(M)/\Gamma_2(P_n(M))$ is a crystallographic group if and only if $M$ is orientable.

If $M$ is orientable, we prove a number of results regarding the structure of $B_n(M)/\Gamma_2(P_n(M))$. We characterise the finite-order elements of this group, and we determine the conjugacy classes of these elements. We also show that there is a single conjugacy class of finite subgroups of $B_n(M)/\Gamma_2(P_n(M))$ isomorphic either to $S_n$ or to certain Frobenius groups. We prove that crystallographic groups whose image by the projection  $B_n(M)/\Gamma_2(P_n(M))\to S_n$  is  a Frobenius group are not Bieberbach groups. Finally, we construct a family of Bieberbach subgroups $\widetilde{G}_{n,g}$ of $B_n(M)/\Gamma_2(P_n(M))$ of dimension $2ng$ and whose holonomy group is the finite cyclic group of order $n$, and if ${\mathcal X}_{n,g}$ is a flat manifold whose  fundamental group is $\widetilde{G}_{n,g}$, we prove that it is an orientable K\"ahler manifold that admits Anosov diffeomorphisms.  
\end{abstract}

\maketitle

\section{Introduction}

The braid groups of the $2$-disc, or Artin braid groups, were introduced by Artin in 1925 and further studied in 1947~\cite{A1,A2}. Surface braid groups were initially studied by Zariski~\cite{Z}, and were later generalised by Fox and Neuwirth to braid groups of arbitrary topological spaces using configuration spaces as follows~\cite{FoN}. Let $M$ be a compact, connected surface, and let $n\in \mathbb N$. The \textit{$n$th ordered configuration space of $M$}, denoted by $F_{n}(M)$, is defined by:
\begin{equation*}
F_n(M)=\left\{(x_{1},\ldots,x_{n})\in M^{n}\,|\,x_{i}\neq x_{j}\,\, \text{if}\,\, i\neq j,\,i,j=1,\ldots,n\right\}.
\end{equation*}
The \textit{$n$-string pure braid group $P_n(M)$ of $M$} is defined by $P_n(M)=\pi_1(F_n(M))$. The symmetric group $S_{n}$ on $n$ letters acts freely on $F_{n}(M)$ by permuting coordinates, and the \textit{$n$-string braid group $B_n(M)$ of $M$} is defined by $B_n(M)=\pi_1(F_n(M)/S_{n})$. This gives rise to the following short exact sequence:
\begin{equation}\label{eq:ses}
1 \to P_{n}(M) \to B_{n}(M) \stackrel{\sigma}{\longrightarrow}S_{n} \to 1.
\end{equation}
The map $\map{\sigma}{B_{n}(M)}[S_{n}]$ is the standard homomorphism that associates a permutation to each element of $S_{n}$.
In~\cite{GGO,GGO2,GGO3}, three of the authors of this paper studied the quotient $B_n/\Gamma_{2}(P_n)$, where $B_{n}$ is the $n$-string Artin braid group, $P_{n}$ is the subgroup of $B_{n}$ of pure braids, and $\Gamma_{2}(P_n)$ is the commutator subgroup of $P_n$. In~\cite{GGO}, it was proved that this quotient is a crystallographic group. Crystallographic groups play an important r\^{o}le in the study of the groups of isometries of Euclidean spaces (see Section~\ref{sec:2} for precise definitions, as well as~\cite{Charlap,Dekimpe,Wolf} for more details). Using different techniques, Marin extended the results of~\cite{GGO} to generalised braid groups associated to arbitrary complex reflection groups~\cite{Ma}. Beck and Marin showed that other finite non-Abelian groups, not covered by~\cite{GGO2,Ma}, embed in  $B_n/\Gamma_{2}(P_n)$~\cite{BM}.

In this paper, we study the quotient $B_n(M)/\Gamma_2(P_n(M))$ of $B_n(M)$, where $\Gamma_2(P_n(M))$ is the commutator subgroup of $P_n(M)$, one of our aims being to decide whether it is crystallographic or not. The group extension~\reqref{ses} gives rise to the following short exact sequence:
\begin{equation}\label{eq:sestorus1}
1\to P_n(M)/\Gamma_2(P_n(M))\to B_n(M)/\Gamma_2(P_n(M))\stackrel{\overline{\sigma}}{\longrightarrow} S_n\to 1.
\end{equation}
Note that if $M$ is an orientable, compact surface of genus $g\geq 1$ without boundary and $n=1$ then $B_1(M)/[P_1(M), P_1(M)]$ is the Abelianisation of $\pi_{1}(M)$, and is isomorphic to $\mathbb{Z}^{2g}$, so it is clearly a crystallographic group.

In Section~\ref{sec:2}, we recall some definitions and facts about crystallographic groups, and if $M$ is an orientable, compact, connected surface of genus $g\geq 1$ without boundary, we prove that $B_n(M)/\Gamma_2(P_n(M))$ is crystallographic.

\begin{prop}\label{prop:toruscryst} 
Let $M$ be an orientable, compact, connected surface of genus $g\geq 1$ without boundary, and let $n\geq 2$. Then there exists a split extension of the form:
\begin{equation}\label{eq:sestorus}
1\to \mathbb{Z}^{2ng}\to B_n(M)/\Gamma_2(P_n(M)) \stackrel{\overline{\sigma}}{\longrightarrow} S_n\to 1,
\end{equation}
where the holonomy representation $\map{\varphi}{S_{n}}[\operatorname{\text{Aut}}(\mathbb{Z}^{2ng})]$ is faithful and where the action is defined by~\reqref{action}. In particular, the quotient $B_n(M)/\Gamma_2(P_n(M))$ is a crystallographic group of dimension $2ng$ and whose holonomy group is $S_n$.
\end{prop}

As for $B_{n}/[P_{n},P_{n}]$, some natural questions arise for $B_n(M)/\Gamma_2(P_n(M))$, such as the existence of torsion, the realisation of elements of finite order and that of finite subgroups, their conjugacy classes, as well as properties of some Bieberbach subgroups of $B_n(M)/\Gamma_2(P_n(M))$. In \reth{ordk0}, we characterise the finite-order elements of $B_{n}(M)/\Gamma_2(P_n(M))$ and their conjugacy classes, from which we see that the conjugacy classes of finite-order elements of $B_{n}(M)/\Gamma_2(P_n(M))$ are in one-to-one correspondence with the conjugacy classes of elements of the symmetric group $S_n$. 

\begin{thm}\label{th:ordk0}
Let $n\geq 2$, and let $M$ be an orientable surface of genus $g\geq 1$ without boundary. 
\begin{enumerate}[(a)]
\item\label{it:ordk0a} Let $e_1$ and $e_2$ be finite-order elements of $B_{n}(M)/\Gamma_2(P_n(M))$. Then $e_1$ and $e_2$ are conjugate if and only if their permutations $\overline{\sigma}(e_1)$ and $\overline{\sigma}(e_2)$ have the same cycle type. Thus two finite cyclic subgroups $H_1$ and $H_2$ of $B_{n}(M)/\Gamma_2(P_n(M))$ are conjugate if and only if the generators of $\overline{\sigma}(H_1)$ and $\overline{\sigma}(H_2)$ have the same cycle type.
\item\label{it:ordk0b} If $H_1$ and $H_2$ are subgroups of $B_{n}(M)/\Gamma_2(P_n(M))$ that are isomorphic to $S_n$ then they are conjugate.
\end{enumerate}
\end{thm}

The results of \reth{ordk0} lead to the following question: if $H_{1}$ and $H_{2}$ are finite subgroups of $B_{n}(M)/\Gamma_2(P_n(M))$ such that $\overline{\sigma}(H_1)$ and $\overline{\sigma}(H_2)$ are conjugate in $S_{n}$, then are $H_{1}$ and $H_{2}$ conjugate? For each odd prime $p$, we shall consider the corresponding Frobenius group, which is the semi-direct product $\mathbb{Z}_p\rtimes \mathbb{Z}_{(p-1)/2}$, the action being given by an automorphism of $\mathbb{Z}_p$ of order $(p-1)/2$. In \repr{embedd} we show that the conclusion of \reth{ordk0} holds for subgroups of $B_5(M)/\Gamma_{2}(P_5(M))$ that are isomorphic to the Frobenius group $\mathbb{Z}_5\rtimes \mathbb{Z}_2$. 

In \resec{3}, we study some Bieberbach subgroups of $B_{n}(M)/\Gamma_2(P_n(M))$ whose construction is suggested by that of the Bieberbach subgroups of $B_{n}/\Gamma_2(P_n)$ given in~\cite{Oc}. 

\begin{thm}\label{th:bieberbach}
Let $n\geq 2$, and let $M$ be an orientable surface of genus $g\geq 1$ without boundary. Let $G_{n}$ be the cyclic subgroup $\ang{(n,n-1,\ldots,2,1)}$ of $S_n$. Then 
there exists a  subgroup $\widetilde{G}_{n,g}$ of $\sigma^{-1}(G_n)/\Gamma_2(P_n(M))\subset B_{n}(M)/\Gamma_2(P_n(M))$ that is a Bieberbach group of dimension $2ng$ whose holonomy group is $G_{n}$. Further, the centre $Z(\widetilde{G}_{n,g})$ of $\widetilde{G}_{n,g}$ is a free Abelian group of rank $2g$.
\end{thm}

The conclusion of the first part of the statement of \reth{bieberbach} probably does not remain valid if we replace the finite cyclic group $G_{n}$ by other finite groups. In this direction, if $p$ is an odd prime, in \repr{frob}, we prove that there is no Bieberbach subgroup $H$ of $B_p(M)/[P_p(M), P_p(M)]$ for which $\overline{\sigma}(H)$ is the Frobenius group $\mathbb{Z}_p\rtimes \mathbb{Z}_{(p-1)/2}$.

It follows from the definition that crystallographic groups act properly discontinuously and cocompactly on Euclidean space, and that the action is free if the groups are Bieberbach. Thus there exists a flat manifold ${\mathcal X}_{n,g}$ whose fundamental group is the subgroup $\widetilde{G}_{n,g}$ of \reth{bieberbach}. Motivated by results about the holonomy representation of Bieberbach subgroups of the Artin braid group quotient $B_{n}/[P_{n}, P_{n}]$ whose holonomy group is a $2$-group obtained in~\cite{OR}, in \resec{3}, we make use of the holonomy representation of $\widetilde{G}_{n,g}$ given in~\reqref{actmatrix} to prove some dynamical and geometric properties of ${\mathcal X}_{n,g}$. To describe these results, we recall some definitions.  

If $\map{f}{M}$ is a self-map of a Riemannian manifold, $M$ is said to have a hyperbolic structure with respect to $f$ if there exists a splitting of the tangent bundle $T(M)$ of the form $T(M)=E^{s} \oplus E^{u}$ such that $\map{\mathrm{D}f}{E^{s}}$ (resp.\ $\map{\mathrm{D}f}{E^{u}}$) is contracting (resp.\ expanding). Further, the map $f$ is called \textit{Anosov} if it is a diffeomorphism and $M$ has hyperbolic structure with respect to $f$. The classification of compact manifolds that admit Anosov diffeomorphisms is a problem first proposed by Smale~\cite{Sm}. Anosov diffeomorphisms play an important r\^ole in the theory of dynamical systems since their behaviour is generic in some sense. Porteous gave a  criterion for the existence of Anosov diffeomorphisms of flat manifolds in terms of the holonomy representation~\cite[Theorems~6.1 and~7.1]{Po} that we shall use in the proof of \reth{anosov}. 

We recall that a K\"ahler manifold is a $2n$-real manifold endowed with a Riemannian metric, a complex structure, and a symplectic structure that is compatible at every point. For more about such manifolds, see~\cite[Chapter 7]{Sz}. A finitely-presented group is said to be a K\"ahler group if it is the fundamental group of a closed  K\"ahler manifold. We may now state \reth{anosov}. 

\begin{thm}\label{th:anosov}
Let $n\geq 2$, and let ${\mathcal X}_{n,g}$ be a $2ng$-dimensional flat manifold whose fundamental group is the Bieberbach group $\widetilde{G}_{n,g}$ of \reth{bieberbach}. Then ${\mathcal X}_{n,g}$ is an orientable K\"ahler manifold with first Betti number $2g$ that admits Anosov diffeomorphisms.
\end{thm}

The proof of \reth{anosov} depends mainly on the holonomy representation of the Bieberbach group $\widetilde{G}_{n,g}$, and makes use of the eigenvalues of the matrix representation and the decomposition of the holonomy representation in irreducible representations using character theory.

Finally, in \resec{4}, we prove in \repr{nonorientable} that the conclusion of \repr{toruscryst} no longer holds if $M$ is the sphere $\mathbb{S}^{2}$ or a compact, non-orientable surface without boundary. More precisely, if $n\geq 1$ then $B_n(M)/\Gamma_2(P_n(M))$ is not a crystallographic group. 

\section{Crystallographic groups and quotients of surface braid groups}\label{sec:2}

In this section, we start by recalling some definitions and facts about crystallographic groups. If $M$ is a compact, orientable surface without boundary of genus $g\geq 1$, in \repr{toruscryst}, we prove that the quotient $B_{n}(M)/\Gamma_2(P_n(M))$ is a crystallographic group that is isomorphic to $\mathbb{Z}^{2ng} \rtimes_{\varphi} S_n$. We also determine the conjugacy classes of the finite-order elements of $B_n(M)/\Gamma_2(P_n(M))$ in \reth{ordk0}.

\subsection{Crystallographic groups}\label{sec:cryst}

In this section, we recall briefly the definitions of crystallographic and Bieberbach groups, and the characterisation of crystallographic groups in terms of a representation that arises in certain group extensions whose kernel is a free Abelian group of finite rank and whose quotient is finite. We also recall some results concerning Bieberbach groups and the fundamental groups of flat Riemannian manifolds. For more details, see~\cite[Section~I.1.1]{Charlap},~\cite[Section~2.1]{Dekimpe} or~\cite[Chapter~3]{Wolf}. 

Let $G$ be a Hausdorff topological group. A subgroup $H$ of $G$ is said to be \emph{discrete} if it is a discrete subset. If $H$ is a closed subgroup of $G$ then the quotient space $G/H$ admits the quotient 
topology for the canonical projection $\map{\pi}{G}[G/H]$, and we say that $H$ is \emph{uniform} if $G/H$ is compact. From now on, we identify $\operatorname{\text{Aut}}(\mathbb{Z}^m)$ with $\operatorname{\text{GL}}(m,\mathbb{Z})$. A discrete, uniform subgroup $\Pi$ of $\R^m\rtimes \operatorname{\text{O}}(m,\R)\subseteq \operatorname{\text{Aff}}(\R^m)$ is said to be a \textit{crystallographic group} of dimension $m$. If in addition $\Pi$ is torsion free then $\Pi$ is called a \textit{Bieberbach group} of dimension $m$.

If $\Phi$ is a group, an \emph{integral representation of rank $m$ of $\Phi$} is defined to be a homomorphism $\map{\Theta}{\Phi}[\operatorname{\text{Aut}}(\mathbb{Z}^m)]$. Two such representations are said to be \emph{equivalent} if their images are conjugate in $\operatorname{\text{Aut}}(\mathbb{Z}^m)$. We say that $\Theta$ is a \emph{faithful representation} if it is injective. We recall the following characterisation of crystallographic groups.

\begin{lem}[{\cite[Lemma~8]{GGO}}]\label{lem:cryst}
Let $\Pi$ be a group. Then $\Pi$ is a crystallographic group if and only if there exists an integer $m\in \mathbb N$, a finite group $\Phi$ and a short exact sequence of the form:
\begin{equation}\label{eq:SeqCrist2}
0\to \mathbb{Z}^m \to \Pi \stackrel{\zeta}{\longrightarrow} \Phi \to 1,
\end{equation}
such that the integral representation $\map{\Theta}{\Phi}[\operatorname{\text{Aut}}(\mathbb{Z}^m)]$ induced by conjugation on $\mathbb{Z}^m$ and defined by $\Theta(\phi)(x)=\pi x \pi^{-1}$ for all $x\in \mathbb{Z}^{m}$ and $\phi\in \Phi$, where $\pi\in \Pi$ is such that $\zeta(\pi)=\phi$, is faithful. 
\end{lem}

If $\Pi$ is a crystallographic group, the integer $m$, the finite group $\Phi$ and the integral representation $\map{\Theta}{\Phi}[\operatorname{\text{Aut}}(\mathbb{Z}^{m})]$ appearing in the statement of~\relem{cryst} are called the \emph{dimension}, the \emph{holonomy group} and the \emph{holonomy representation} of  $\Pi$ respectively.

We now recall the connection between Bieberbach groups and manifolds. A Riemannian manifold $M$ is called \emph{flat} if it has zero curvature at every point. By the first Bieberbach Theorem, there is a correspondence between Bieberbach groups and fundamental groups of flat Riemannian manifolds without boundary (see~\cite[Theorem~2.1.1]{Dekimpe} and the paragraph that follows it). By~\cite[Corollary~3.4.6]{Wolf}, the holonomy group of a flat manifold $M$ is isomorphic to the group $\Phi$. In 1957, Auslander and Kuranishi proved that every finite group is the holonomy group of some flat manifold~(see~{\cite[Theorem~3.4.8]{Wolf} and~\cite[Theorem~III.5.2]{Charlap}}). It is well known that a flat manifold determined by a Bieberbach group $\Pi$ is orientable if and only if the integral representation $\map{\Theta}{\Phi}[\operatorname{\text{GL}}(m,\mathbb{Z})]$ satisfies $\operatorname{\text{Im}}(\Theta) \subseteq \operatorname{\text{SL}}(m,\mathbb{Z})$~\cite[Theorem~6.4.6 and Remark~6.4.7]{Dekimpe}.  This being the case, $\Pi$ is said to be an \emph{orientable Bieberbach group}.  

\subsection{The  group $B_n(M)/\Gamma_2(P_n(M))$}

Let $M$ be a compact, orientable surface without boundary of genus $g\geq 1$. Besides showing that the group $B_n(M)/\Gamma_2(P_n(M))$ is crystallographic, we shall also be interested in the conjugacy classes of its elements by elements of $P_n(M)/\Gamma_2(P_n(M))$, as well as the conjugacy classes of its finite subgroups. In order to study these questions, it is useful to have an algebraic description of this quotient at our disposal. We will make use of the presentations of the (pure) braid groups of $M$ given in~\cite[Theorems~2.1 and~4.2]{GM}, where for all $1\leq i<j\leq n$, $1\leq r\leq 2g$ and $1\leq k\leq n$, the elements $a_{k,r}$ and $T_{i,j}$ in $B_n(M)$ are described in~\cite[Figure~9]{GM}, and for all $1\leq i\leq n-1$, the elements $\sigma_i$ are the classical generators of the Artin braid group that satisfy the \emph{Artin relations}:
\begin{equation}\label{eq:artin}
\begin{cases}
\sigma_{i}\sigma_{j}=\sigma_{j}\sigma_{i} & \text{for all $1\leq i,j\leq n-1$, $\left\lvert i-j\right\rvert \geq 2$}\\
\sigma_{i}\sigma_{i+1}\sigma_{i}=\sigma_{i+1}\sigma_{i}\sigma_{i+1} & \text{for all $1\leq i\leq n-2$.}
\end{cases}
\end{equation}
We recall that the full twist braid of $B_n(M)$, denoted by $\Delta_{n}^{2}$, is defined by:
\begin{equation}\label{eq:deffulltwist}
\Delta_{n}^{2}=(\sigma_{1}\cdots\sigma_{n-1})^{n},
\end{equation}
and is equal to: 
\begin{equation}\label{eq:fulltwist}
\Delta^{2}_{n}=A_{1,2}(A_{1,3}A_{2,3})\cdots(A_{1,n}A_{2,n}\cdots A_{n-1,n}),
\end{equation}
where for $1\leq i<j\leq n$, the elements $A_{i,j}$ are the usual Artin generators of $P_{n}$ defined by $A_{i,j}=\sigma_{j-1}\cdots\sigma_{i+1}\sigma^{2}_{i}\sigma^{-1}_{i+1}\cdots\sigma^{-1}_{j-1}$. By abuse of notation, in what follows, if $\alpha\in B_{n}(M)$, we also denote its $B_n(M)/\Gamma_2(P_n(M))$-coset by $\alpha$. The following proposition gives some relations in $B_n(M)$ that will be relevant to our study of $B_n(M)/\Gamma_2(P_n(M))$. 

\begin{prop}\label{prop:subs}
Let $M$ be a compact, orientable surface without boundary of genus
$g\geq 1$, let $1\leq i\leq n-1$, $1\leq j\leq n$ and $1\leq r\leq
2g$, and let $\widetilde{A}_{j,r}=a_{j,1}\cdots
a_{j,r-1}a^{-1}_{j,r+1}\cdots a^{-1}_{j,2g}$. The following relations
hold in $B_{n}(M)$:
\begin{enumerate}[(a)]
\item\label{it:subsa}\label{subs:1} $\sigma_{i}a_{j,r}\sigma^{-1}_{i}=
\begin{cases}
a_{i+1,r}\sigma^{-2}_{i} & \text{if $j=i$ and $r$ is even}\\
\sigma^{2}_{i}a_{i+1,r} & \text{if $j=i$ and $r$ is odd}\\
\sigma^{2}_{i}a_{i} & \text{if $j=i+1$ and $r$ is even}\\
a_{i,r}\sigma^{-2}_{i} & \text{if $j=i+1$ and $r$ is odd}\\
a_{j,r} & \text{if $j\neq i, i+1$.}
\end{cases}$
\item\label{subs:3}
$T_{i,j}=\sigma_{i}\sigma_{i+1}\cdots\sigma_{j-2}\sigma^{2}_{j-1}\sigma_{j-2}\cdots\sigma_{i}$ where $1\leq i,j\leq n$ and $i+1<j$, and $T_{i,i+1}=\sigma_{i}^2$  for all $1\leq i\leq n-1$.

\item\label{subs:4} $T_{i,j}= [a_{i,1}\cdots
a_{i,r},\widetilde{A}_{j,r}]\, T_{i,j-1}$, for all $1\leq i <j \leq n$
and $1\leq r \leq 2g$.

\item\label{subs:5} For all $1\leq i \leq j \leq n$, $T_{i,j}$ and $\Delta^{2}_{n}$ belong to $\Gamma_{2}(P_{n}(M))$.
\end{enumerate}
\end{prop}

\begin{proof}
Part~(\ref{subs:1}) is a consequence of relations~(R7) and~(R8) of \cite[Theorem~4.2, step~3]{GM}, with the exception of the case $j\neq i, i+1$, which is clear. Part~(\ref{subs:3}) is relation~(R9) of \cite[Theorem~4.2, step~3]{GM}, and part~(\ref{subs:4}) is relation~(PR3) of \cite[Theorem 4.2, presentation~1]{GM}. By~\cite[page~439]{GM}, $T_{j-1,j-1}=1$ for all $1<j\leq n+1$, and it follows from part~(\ref{subs:4}) that  $T_{j-1,j} \in \Gamma_{2}(P_{n}(M))$ for all $2\leq j\leq n$, and then by induction on $j-i$ that $T_{i,j}\in\Gamma_2(P_n(M))$ for all $1\leq i \leq  j \leq n$. Using the Artin relations~\reqref{artin} and part~(\ref{subs:3}), we have $A_{i,j}=T^{-1}_{i,j-1}T_{i,j}$ for all $1\leq i< j \leq n$, so $A_{i,j}$ also belongs to $\Gamma_{2}(P_{n}(M))$ by part~(\ref{subs:4}), and thus $\Delta_{n}^{2}\in \Gamma_{2}(P_{n}(M))$ by~\reqref{fulltwist}. 
\end{proof}

This allows us to compute the Abelianisation of $P_n(M)$.

\begin{cor}\label{cor:abel}
Let $M$ be a compact, orientable surface without boundary of genus $g\geq 1$, and let $n\geq 1$. Then the Abelianisation of $P_n(M)$ is a free Abelian group of rank $2ng$, for which $\{ a_{i,r} \,\mid\, \text{$i=1,\ldots,n$ and $r=1,\ldots, 2g$}\}$ is a basis.
\end{cor}

\begin{proof}
The result follows from the presentation of $P_{n}(M)$ given in~\cite[Theorem~4.2]{GM} and the fact that for all $1\leq i< j \leq n$, $T_{i,j}\in \Gamma_{2}(P_{n}(M))$ by \repr{subs}(\ref{subs:5}).
\end{proof}

For all $1\leq i\leq n-1$, we have $\overline{\sigma}(\sigma_{i})=\tau_{i}$, where $\tau_{i}$ denotes the transposition $(i,i+1)$ in $S_{n}$. Using \repr{subs}(\ref{it:subsa}), and identifying $\mathbb{Z}^{2ng}$  with $P_n(M)/\Gamma_2(P_n(M))$ via \reco{abel}, we obtain the induced action $\map{\varphi}{S_n}[\operatorname{\text{Aut}}(\mathbb{Z}^{2ng})]$, that for all $1\leq i\leq n-1$, $1\leq j\leq n$ and $1\leq r\leq 2g$, is defined by: 
\begin{equation}\label{eq:action}
\phi(\tau_{i})(a_{j,r})=\sigma_{i}a_{j,r}\sigma^{-1}_{i}= a_{\tau_{i}(j),r}.
\end{equation}

The following result is the analogue of~\cite[Proposition~12]{GGO} for braid groups of orientable surfaces.

\begin{prop}\label{prop:action}
Let $M$ be a compact, orientable surface without boundary of genus $g\geq 1$, and let $n\geq 1$. Let $\alpha \in B_n(M)/\Gamma_2(P_n(M))$, and let $\pi =\overline{\sigma}(\alpha^{-1})$. Then $\alpha a_{i,r}\alpha^{-1}=a_{\pi(i),r}$ in $B_n(M)/\Gamma_2(P_n(M))$ for all $1\leq i\leq n$ and $1\leq r\leq 2g$.
\end{prop}

\begin{proof}
The proof is similar to that of~\cite[Proposition~12]{GGO}, and makes use of~\reqref{action}. The details are left to the reader.
\end{proof}

We now give a presentation of $B_n(M)/\Gamma_2(P_n(M))$. 

\begin{prop}\label{prop:pres_quo}
Let $M$ be a compact, orientable surface without boundary of genus $g\geq 1$, and let $n\geq 1$. The quotient group $B_n(M)/\Gamma_2(P_n(M))$ has the following presentation:

\noindent
Generators: $\sigma_{1},\ldots, \sigma_{n-1}, a_{i,r},\,1\leq i \leq n,\, 1\leq r \leq 2g$.

\noindent
Relations:
\begin{enumerate}[(a)]
\item\label{it:pq1} the Artin relations~\reqref{artin}.
	
\item\label{it:pq2} $\sigma^{2}_{i}=1$, for all $i=1,\ldots, n-1$.
	
\item\label{it:pq3} $[a_{i,r},a_{j,s}]=1$, for all $i,j=1,\ldots,n$ and $r,s=1,\ldots, 2g$.
	
\item\label{it:pq4} $\sigma_{i}a_{j,r}\sigma^{-1}_{i}= a_{\tau_{i}(j),r}$ for all $1\leq i\leq n-1$, $1\leq j\leq n$ and $1\leq r\leq 2g$. 
\end{enumerate}
\end{prop}

\begin{proof}
The given presentation of $B_n(M)/\Gamma_2(P_n(M))$ may be obtained by applying the standard method for obtaining a presentation of a group extension~\cite[Proposition~1, p.~139]{Johnson} to the short exact sequence~\reqref{sestorus1}  for $M$ satisfying the hypothesis, and using \reco{abel}, the equality $\overline{\sigma}(\sigma_{i})=\tau_{i}$ for all $i=1,\ldots,n-1$, and the fact that the relations~(\ref{it:pq1}) and~(\ref{it:pq2}) constitute a presentation of $S_{n}$ for the generating set $\{ \tau_{1},\ldots, \tau_{n-1}\}$.
\end{proof}

We may now prove \repr{toruscryst}.

\begin{proof}[Proof of \repr{toruscryst}]
Assume that $n\geq 2$. The short exact sequence~\reqref{sestorus} is obtained from~\reqref{sestorus1} using \reco{abel}. To prove that the short exact sequence~\reqref{sestorus} splits,  let $\map{\psi}{S_n}[B_n(M)/\Gamma_2(P_n(M))]$ be the map defined on the generating set $\{ \tau_{1},\ldots, \tau_{n-1}\}$ of $S_{n}$ by $\psi(\tau_{i})=\sigma_i$ for all $i=1,\ldots, n-1$. Relations~(\ref{it:pq1}) and~(\ref{it:pq2}) of \repr{pres_quo} imply that $\psi$ is a homomorphism.
Consider the action $\map{\varphi}{S_n}[\operatorname{\text{Aut}}(\mathbb{Z}^{2ng})]$ defined by~\reqref{action}. By \repr{action}, $\varphi(\theta)$ is the identity automorphism if and only if $\theta$ is the trivial permutation, from which it follows that $\varphi$ is injective. 
The rest of the statement of \repr{toruscryst} is a consequence of \relem{cryst}.
\end{proof}

\begin{cor}\label{cor:subgroup}
Let $M$ be a compact, orientable surface without boundary of genus $g\geq 1$, let $n\geq 2$, and let $H$ be a subgroup of $S_n$. Then the group ${\sigma}^{-1}(H)/\Gamma_2(P_n(M))$ is a crystallographic group of dimension $2ng$ whose holonomy group is $H$.
\end{cor}

\begin{proof}
The result is an immediate consequence of \repr{toruscryst} and \cite[Corollary~10]{GGO}.
\end{proof}

We now turn to the proof of \reth{ordk0}. For this, we will require the following lemma.

\begin{lem}\label{lem:coeftjr}
Let $M$ be a compact, orientable surface without boundary of genus $g\geq 1$, and let $n\geq 1$. Let $z\in B_{n}(M)/\Gamma_2(P_n(M))$. Let $z=\omega\prod^{n}_{i=1}\prod^{2g}_{r=1}a^{s_{i,r}}_{i,r}\in B_{n}(M)/\Gamma_2(P_n(M))$, where $\omega=\psi(\overline{\sigma}(z))$, and $s_{i,r}\in \Z$ for all $1\leq i\leq n$ and $1\leq r\leq 2g$. Suppose that $\overline{\sigma}(z)$ is the $m$-cycle $(l_1,\ldots,l_m)$, where $m\geq 2$, and let $k\in \N$ be such that $m$ divides $k$. Then $z^k=\prod^{n}_{i=1}\prod^{2g}_{r=1}a_{i,r}^{t_{i,r}}$ where for all $1\leq i\leq n$ and $1\leq r\leq 2g$, $t_{i,r}\in \Z$ is given by:
\begin{equation}\label{eq:deftij}
t_{i,r}=\begin{cases}ks_{i,r} & \text{if $i\notin \{l_1,\ldots,l_m\}$}\\
\frac{k}{m}\Sigma_{j=1}^{m} s_{l_j,r}& \text{if $i=l_{j}$, where $1\leq j\leq m$.}
\end{cases}
\end{equation}
\end{lem}

\begin{proof} 
Let $z\in B_{n}(M)/\Gamma_2(P_n(M))$ be as in the statement, let $\overline{\sigma}(z)=(l_1,\ldots,l_m)$, where $m\geq 2$, and let $\omega=\psi(\overline{\sigma}(z))$. By \repr{toruscryst}, $\omega$ is of order $m$, and the decomposition $z=\omega\prod^{n}_{i=1}\prod^{2g}_{r=1}a^{s_{i,r}}_{i,r}$ arises from~\reqref{sestorus}. By \reco{abel} and \repr{action} and using the fact that $\omega^{k}=1$, we obtain: 
\begin{align}
z^{k}&= \left( \omega\prod^{n}_{i=1}\prod^{2g}_{r=1}a^{s_{i,r}}_{i,r}\right)^{k}= \omega^{k} \left[ \prod_{j=1}^{k} \omega^{j-k} \left( \prod^{n}_{i=1} \prod^{2g}_{r=1}a^{s_{i,r}}_{i,r} \right)\omega^{k-j} \right]= \prod_{j=1}^{k} \prod^{n}_{i=1}\prod^{2g}_{r=1}a^{s_{i,r}}_{\overline{\sigma}(\omega^{k-j})(i),r}\notag\\
&=  \prod^{n}_{i=1}\prod^{2g}_{r=1} \prod_{j=1}^{k} a^{s_{\overline{\sigma}(\omega^{j-k})(i),r}}_{i,r}= \prod^{n}_{i=1}\prod^{2g}_{r=1} a^{t_{i,r}}_{i,r},\label{eq:thetak}
\end{align}
where $t_{i,r}= \sum_{j=1}^{k} s_{\overline{\sigma}(\omega^{j-k})(i),r}$ for all $1\leq i\leq n$ and $1\leq r\leq 2g$. Equation~\reqref{deftij} then follows, using the fact that $\overline{\sigma}(\omega^{-u})(l_i)=l_{i-u}$ for all $u\in \Z$, where the index $i-u$ is taken modulo $m$.
\end{proof}

We now prove \reth{ordk0}.

\begin{proof}[Proof of \reth{ordk0}]\mbox{}
\begin{enumerate}[(a)]
\item Let $\map{\psi}{S_n}[B_n(M)/\Gamma_2(P_n(M))]$ be the section for the short exact sequence~\reqref{sestorus} given in the proof of Proposition \ref{prop:toruscryst}, and let $e_{1}$ and $e_{2}$ be finite-order elements of $B_n(M)/\Gamma_2(P_n(M))$. If $e_{1}$ and $e_{2}$ are conjugate in $B_n(M)/\Gamma_2(P_n(M))$ then $\overline{\sigma}(e_1)$ and $\overline{\sigma}(e_2)$ are conjugate in $S_{n}$, and so have the same cycle type. Conversely, suppose that the permutations $\overline{\sigma}(e_1)$ and $\overline{\sigma}(e_2)$ have the same cycle type. Then they are conjugate in $S_{n}$, so there exists $\xi\in S_{n}$ such that $\overline{\sigma}(e_1)=\xi \overline{\sigma}(e_2) \xi^{-1}$, and up to substituting $e_{2}$ by $\psi(\xi^{-1}) e_{2} \psi(\xi)$ if necessary, we may assume that $\overline{\sigma}(e_1)=\overline{\sigma}(e_2)$. We claim that if $\theta$ is any finite-order element of $B_n(M)/\Gamma_2(P_n(M))$ then $\theta$ and $\psi(\overline{\sigma}(\theta))$ are conjugate in $B_n(M)/\Gamma_2(P_n(M))$. This being the case, for $i=1,2$, $e_{i}$ is conjugate to $\psi(\overline{\sigma}(e_{i}))$, but since $\psi(\overline{\sigma}(e_{1}))=\psi(\overline{\sigma}(e_{2}))$, it follows that $e_{1}$ and $e_{2}$ are conjugate as required, which proves the first part of the statement. To prove the claim, set $\theta=\omega\prod^{n}_{i=1}\prod^{2g}_{r=1}a^{s_{i,r}}_{i,r}\in B_{n}(M)/\Gamma_2(P_n(M))$ as in the proof of \relem{coeftjr}, where $\omega= \psi(\overline{\sigma}(\theta))$. It thus suffices to prove that $\theta$ and $\omega$ are conjugate in $B_n(M)/\Gamma_2(P_n(M))$. Let $\overline{\sigma}(\theta)= \tau_{1}\cdots \tau_{d}$ be the cycle decomposition of $\overline{\sigma}(\theta)$, where for $i=1,\ldots,d$, $\tau_{i}=(l_{i,1},\ldots,l_{i,m_{i}})$ is an $m_{i}$-cycle, and $m_{i}\geq 2$, and let $k=\operatorname{\text{lcm}}(m_{1},\ldots,m_{d})$ be the order of $\overline{\sigma}(\theta)$, which is also the order of $\theta$ and of $\omega$. For $t=1,\ldots,n$ and $r=1,\ldots, 2g$, we define $p_{t,r}\in \Z$ as follows. If $t$ does not belong to the support $\operatorname{\text{Supp}}(\overline{\sigma}(\theta))$ of $\overline{\sigma}(\theta)$, let $p_{t,r}=0$. If $t\in \operatorname{\text{Supp}}(\overline{\sigma}(\theta))$ then $t=l_{i,q}$ for some $1\leq i\leq d$ and $1\leq q \leq m_{i}$, and we define $p_{t,r}= -\sum_{j=1}^{q} s_{l_{i,j},r}$. It follows from~\relem{coeftjr} and the fact that $\theta$ is of order $k$ that $p_{l_{i,m_{i}},r}=0$, for all $i=1,\ldots,d$. Then for all $1\leq i\leq d$, $2\leq q\leq m_{i}$ and $1\leq r\leq 2g$, we have:
\begin{equation}\label{eq:diffp}
\text{$p_{l_{i,q-1},r}-p_{l_{i,q},r}=s_{l_{i,q},r}$ and $p_{l_{i,m_{i}},r}-p_{l_{i,1},r}=-p_{l_{i,1},r}=s_{l_{i,1},r}$.}
\end{equation}
Let $\alpha=\prod^{n}_{i=1}\prod^{2g}_{r=1}a^{p_{i,r}}_{i,r}\in P_{n}(M)/\Gamma_2(P_n(M))$. By \reco{abel} and \repr{action}, we have:
\begin{align*}
\alpha \omega \alpha^{-1} &= \omega\ldotp \omega^{-1} \left(\prod^{n}_{i=1}\prod^{2g}_{r=1}a^{p_{i,r}}_{i,r} \right)\omega \ldotp \prod^{n}_{i=1}\prod^{2g}_{r=1}a^{-p_{i,r}}_{i,r} = \omega \left( \prod^{n}_{i=1}\prod^{2g}_{r=1}a^{p_{\overline{\sigma}(\omega^{-1})(i),r}}_{i,r} \right) \ldotp \prod^{n}_{i=1}\prod^{2g}_{r=1}a^{-p_{i,r}}_{i,r}\\
&= \omega\prod^{n}_{i=1}\prod^{2g}_{r=1}a^{p_{\overline{\sigma}(\omega^{-1})(i),r}-p_{i,r}}_{i,r}=\omega \prod^{n}_{i=1}\prod^{2g}_{r=1}a^{s_{i,r}}_{i,r}=\theta,
\end{align*}
where we have also made use of~\reqref{diffp}. Thus $\theta$ is conjugate to $\omega$, which proves the claim, and thus the first part of the statement. 

\item We start by showing that if $H\subset B_n(M)/\Gamma_2(P_n(M))$ is isomorphic to $S_n$ then $H$ and $\psi(\overline{\sigma}(H))$ are conjugate. This being the case, it follows that each of the subgroups $H_1, H_2$ is conjugate to $\psi(\overline{\sigma}(H))$, and the result follows. Suppose that $H$ is a group isomorphic to $S_{n}$ that embeds in $(\mathbb{Z}^n \oplus\cdots\oplus \mathbb{Z}^n)\rtimes S_n$. Using the $\mathbb{Z}[S_n]$-module structure of $\mathbb{Z}^{2ng}$ given above, it follows that $H$ embeds in $\mathbb{Z}^{n}\rtimes S_n$, for any one of $2g$ summands of $\mathbb{Z}^{n}$. Let us first prove that the result holds for such an embedding. Let $\map{s}{H}[\mathbb{Z}^{n}\rtimes S_n]$ be an embedding, let $\map{\pi}{\mathbb{Z}^{n}\rtimes S_n}[S_{n}]$ be projection onto the second factor, and let $\map{\psi}{S_{n}}[\mathbb{Z}^{n}\rtimes S_n]$ be inclusion into the second factor. Since $\ker{\pi}=\Z^{n}$ is torsion free,  the restriction of $\pi$ to $s(H)$ is injective, and so $\map{\pi\circ s}{H}[S_{n}]$ is an isomorphism. Let us prove that the subgroups $s(H)$ and $S_{n}$ of $\mathbb{Z}^{n}\rtimes S_n$ are conjugate. Let $\brak{\tau_1,\ldots,\tau_{n-1}}$ be the generating set of $S_{n}$ defined previously, and for $i=1,\ldots,n-1$, let $\alpha_{i}$ be the unique element of $H$ for which $\pi\circ s(\alpha_{i})=\tau_{i}$. Then $H$ is generated by $\brak{\alpha_1,\ldots,\alpha_{n-1}}$, subject to the Artin relations and $\alpha_{i}^{2}=1$ for all $i=1,\ldots,n-1$. There exist $a_{i,j}\in \Z$, where $1\leq i\leq n-1$ and $1\leq j\leq n$, such that $s(\alpha_{i})=(a_{i,1},\ldots,a_{i,n})\tau_{i}$ in $\mathbb{Z}^{n}\rtimes S_n$. Using the action of $S_{n}$ on $\Z^{n}$ and the fact that $\alpha_{i}^{2}=1$, it follows that $a_{i,j}=0$ for all $j\neq i,i+1$, and $a_{i,i+1}=-a_{i,i}$. Then $s(\alpha_{i})=(0,\ldots,0,a_i,-a_i,0,\ldots,0)\tau_i$ for all $1\leq i\leq n-1$, where  $a_i=a_{i,i}$, and the element $a_{i}$ is in the $i$th position. One may check easily that these elements also satisfy the Artin relations in $\mathbb{Z}^{n}\rtimes S_n$. Let $x_{1}\in \Z$, and for $i=2,\ldots, n$, let $x_{i}= x_{1}-\sum_{j=1}^{i-1}a_j$. Thus $x_{i}-x_{i+1}=a_{i}$ for all $i=1,\ldots, n-1$, and so $(x_1,x_2,\ldots, x_{i},x_{i+1},\ldots, x_n)\tau_i(-x_1,-x_2,\ldots,-x_{i},-x_{i+1},\ldots-x_n)= (0,\ldots,0,a_i,-a_i,0, \ldots,0) \tau_i= s(\alpha_{i})$. We conclude that the subgroups $s(H)$ and $S_{n}$ of $\mathbb{Z}^{n}\rtimes S_n$ are conjugate. 

Returning to the general case where $H$ embeds in $(\mathbb{Z}^n \oplus\cdots\oplus \mathbb{Z}^n)\rtimes S_n$, the previous paragraph shows that the embedding of $H$ into each $\mathbb{Z}^{n}\rtimes S_n$ is conjugate by an element of the same factor $\mathbb{Z}^{n}$ to the factor $S_{n}$. The result follows by conjugating by the element of $\mathbb{Z}^n \oplus\cdots\oplus \mathbb{Z}^n$ whose $i$th factor is the conjugating element associated to the $i$th copy of $\mathbb{Z}^{n}\rtimes S_n$  for all $i=1,\ldots,2g$.\qedhere
\end{enumerate}
\end{proof}

With the statement of \reth{ordk0} in mind, one may ask whether the result of the claim extends to other finite subgroups. More precisely, if $G$ is a finite subgroup of $B_n(M)/\Gamma_2(P_n(M))$, are $G$ and $\psi(\overline{\sigma}(G))$ conjugate? We have a positive answer in the following case. 

\begin{prop}\label{prop:embedd}
Let $M$ be a compact, orientable surface without boundary of genus $g\geq 1$. If $H_1$ and $H_2$ are subgroups of $B_{5}(M)/\Gamma_2(P_5(M))$ that are isomorphic to the Frobenius group $\mathbb{Z}_5\rtimes \mathbb{Z}_{2}$ then they are conjugate.
\end{prop}

\begin{proof}
Using \repr{toruscryst}, we identify $B_5(M)/\Gamma_2(P_5(M))$ with $\mathbb{Z}^{10g}\rtimes S_5$. As in the proof of \reth{ordk0}(\ref{it:ordk0b}), we decompose the first factor as a direct sum $\mathbb{Z}^{10g}=\mathbb{Z}^5 \oplus\cdots\oplus \mathbb{Z}^5$ of $2g$ copies of $\mathbb{Z}^5$, which we interpret as a $\mathbb{Z}[S_5]$-module, the module structure being given by \repr{action}.

Let $H$ be a group isomorphic to a subgroup of $S_{5}$ that embeds in $(\mathbb{Z}^5 \oplus\cdots\oplus \mathbb{Z}^5)\rtimes S_5\cong  B_5(M)/\Gamma_2(P_5(M))$.   Using the $\mathbb{Z}[S_5]$-module structure of $\mathbb{Z}^{10g}$ given above, it follows that $H$ embeds in $\mathbb{Z}^{5}\rtimes S_5$, for any one of the $2g$ summands of $\mathbb{Z}^{5}$. We will first prove the statement for the embedding of the Frobenius group $\mathbb{Z}_5\rtimes \mathbb{Z}_{2}$ in $\mathbb{Z}^{5}\rtimes S_5$, and then deduce the result in the general case. Let $H$ be this Frobenius group, let $\map{s}{H}[\mathbb{Z}^{5}\rtimes S_5]$ be an embedding, let $\map{\pi}{\mathbb{Z}^{5}\rtimes S_5}[S_{5}]$ be projection onto the first factor, and let $\map{\psi}{S_{5}}[\mathbb{Z}^{5}\rtimes S_5]$ be inclusion into the second factor. Since $\ker{\pi}=\Z^{5}$ is torsion free,  the restriction of $\pi$ to $s(H)$ is injective, and so $\map{\pi\circ s}{H}[S_{5}]$ is an embedding of $H$ into $S_{5}$. Let us prove that the subgroups $s(H)$ and $\psi\circ \pi\circ s(H)$ of $\mathbb{Z}^{5}\rtimes S_5$ are conjugate. First observe that the Frobenius group $\mathbb{Z}_5\rtimes \mathbb{Z}_2$ embeds in $S_{5}$ by sending a generator $\iota_5$ of $ \mathbb{Z}_5$ to the permutation $w_1=(1,2,3,4,5)$ and the generator $\iota_2$ of $\mathbb{Z}_2$ to $w_2=(1,4) (2,3)$. The group $S_{5}$ contains six cyclic subgroups of order $5$ that are  conjugate to $\ang{w_{1}}$, from which we deduce the existence of six pairwise conjugate subgroups of $S_{5}$ isomorphic to $\mathbb{Z}_5\rtimes \mathbb{Z}_2$, each of which contains one of the cyclic subgroups of order $5$. We claim that these are exactly the subgroups of $S_{5}$ isomorphic to $\mathbb{Z}_5\rtimes \mathbb{Z}_2$. To see this, let $K$ be such a subgroup. Then the action of an element $k$ of $K$ of order $2$ on a $5$-cycle $(a_{1},\ldots,a_{5})$ of $K$ inverts the order of the elements $a_{1},\ldots,a_{5}$, and this can only happen if $k$ is a product of two disjoint transpositions. 
Further, there are exactly five products of two disjoint transpositions whose action by conjugation on $(a_{1},\ldots,a_{5})$ inverts the order of the elements $a_{1},\ldots,a_{5}$, and these are precisely the elements of $K$ of order $2$. In particular, each cyclic subgroup of $S_{5}$ of order $5$ is contained in exactly one subgroup of $S_{5}$ isomorphic to $\mathbb{Z}_5\rtimes \mathbb{Z}_2$. This proves the claim. So up to composing $\pi$ by an inner automorphism of $S_{5}$ if necessary, we may assume that $\pi\circ s(H)=\ang{w_1,w_2}$. Applying methods similar to those of the proof of \relem{coeftjr}, the relations $\iota_5^5=1$ and $\iota_2^2=1$ imply that there exist $a_1,\ldots,a_5 ,x, y\in \Z$ such that:
\begin{equation}\label{eq:imsw}
\text{$s(w_1)=(a_1, a_2, a_3, a_4, -a_1-a_2-a_3-a_4)w_1$ and $s(w_2)=(x, y, -y, -x, 0)w_2$.}
\end{equation}
Taking the image of the relation $w_2w_1w_2^{-1}=w_1^{-1}$ by $s$, using~\reqref{imsw} and simplifying the resulting expression, we obtain:
\begin{equation}\label{eq:eqxy}
\text{$x=-a_2-a_3-a_4$ and $y=-a_3$.}
\end{equation}
Any map $\map{s}{H}[\mathbb{Z}^{5}\rtimes S_5]$ of the form~\reqref{imsw} for which the relations~\reqref{eqxy} hold gives rise to an embedding of $H$ in $\mathbb{Z}^{5}\rtimes S_5$. We claim that the image of the embedding is conjugate to the group $\ang{w_{1},w_{2}}$ (viewed as a subgroup of the second factor of $\mathbb{Z}^{5}\rtimes S_5$). To do so, let $\map{s}{H}[\mathbb{Z}^{5}\rtimes S_5]$ of the form~\reqref{imsw} for which the relations~\reqref{eqxy} hold. We will show that there exists $a\in \mathbb{Z}^5$ such that $aw_{i}a^{-1}=s(w_{i})$ for $i=1,2$. Let $\lambda_{5}\in \Z$, and for $i=1,\ldots,4$, let $\lambda_{i}=\lambda_{5}+\sum_{j=1}^{i} a_{i}$, and let $a=(\lambda_1,\lambda_2, \lambda_3, \lambda_4, \lambda_5)\in \mathbb{Z}^5$. Then in $\mathbb{Z}^{5}\rtimes S_5$, using~\reqref{imsw} and~\reqref{eqxy}, we have:
\begin{align*}
aw_{1}a^{-1}&= (\lambda_1,\lambda_2, \lambda_3, \lambda_4, \lambda_5) w_{1}  (-\lambda_1, -\lambda_2, -\lambda_3, -\lambda_4, -\lambda_5)\\
&= (\lambda_1,\lambda_2, \lambda_3, \lambda_4, \lambda_5) \ldotp (-\lambda_5,-\lambda_1, -\lambda_2, -\lambda_3, -\lambda_4)w_{1}\\
&= (\lambda_1-\lambda_5,\lambda_2-\lambda_1, \lambda_3-\lambda_2, \lambda_4-\lambda_3, \lambda_5-\lambda_4)w_{1}\\
&=(a_{1},a_{2},a_{3},a_{4},-a_1-a_2-a_3-a_4) w_{1}=s(w_{1}),
\end{align*}
and
\begin{align*}
aw_{2}a^{-1}&=(\lambda_1,\lambda_2, \lambda_3, \lambda_4, \lambda_5) w_{2}  (-\lambda_1, -\lambda_2, -\lambda_3, -\lambda_4, -\lambda_5)\\
&= (\lambda_1,\lambda_2, \lambda_3, \lambda_4, \lambda_5)  \ldotp (-\lambda_4, -\lambda_3, -\lambda_2, -\lambda_1, -\lambda_5)w_{2}\\
&= (\lambda_1-\lambda_4, \lambda_2-\lambda_3, \lambda_3-\lambda_2, \lambda_4-\lambda_1, 0) w_{2}\\
&= (-a_2-a_3-a_4, -a_{3},a_{3}, a_2+a_3+a_4,0)w_{2}=(x,y,-y,-x,0)w_{2}=s(w_{2}).
\end{align*}
It follows that the subgroups $s(H)$ and $\ang{w_{1},w_{2}}$ are conjugate in $\mathbb{Z}^{5}\rtimes S_5$.

As in the proof of \reth{ordk0}(\ref{it:ordk0b}), returning to the general case where $H$ embeds in $(\mathbb{Z}^5 \oplus\cdots\oplus \mathbb{Z}^5)\rtimes S_5$, the previous paragraph shows that the embedding of $H$ into each $\mathbb{Z}^{5}\rtimes S_5$ is conjugate by an element of the same factor $\mathbb{Z}^{5}$ to the factor $S_{5}$. The result follows by conjugating by the element of $\mathbb{Z}^5 \oplus\cdots\oplus \mathbb{Z}^5$ whose $i$th factor is the conjugating element associated to the $i$th copy of $\mathbb{Z}^{5}\rtimes S_5$  for all $i=1,\ldots,2g$.
\end{proof}

\section{Some Bieberbach subgroups of $B_n(M)/\Gamma_2(P_n(M))$ and K\"ahler flat manifolds}\label{sec:3}

By \reco{subgroup}, for any subgroup $H$ of $S_n$, the quotient group $\sigma^{-1}(H)/\Gamma_2(P_n(M))$ is a crystallographic group that is not Bieberbach since it has torsion elements. We start by proving \reth{bieberbach}, which states that $B_n(M)/\Gamma_2(P_n(M))$ admits Bieberbach subgroups. More precisely, for all $n\geq 2$, we will consider the cyclic subgroup $G_{n}=\ang{(n,n-1,\ldots,2,1)}$ of $S_n$, and we show that $\sigma^{-1}(G_n)/\Gamma_2(P_n(M))$ admits a Bieberbach subgroup $\widetilde{G}_{n,g}$ of dimension $2ng$ whose holonomy group is $G_n$. The group $\widetilde{G}_{n,g}$ is thus the fundamental group of a flat manifold. In \reth{anosov}, we will prove that this flat manifold is orientable and admits a K\"ahler structure as well as Anosov diffeomorphisms.

\begin{proof}[Proof of \reth{bieberbach}]
Let $\alpha_{n-1}=\sigma_1\cdots \sigma_{n-1} \in B_n(M)/\Gamma_2(P_n(M))$. By~\req{sestorus}, $\overline{\sigma}(\alpha_{n-1})$ generates the subgroup $G_n$ of $S_{n}$. Further, the full twist of $B_n(M)$ is a coset representative of $\alpha_{n-1}^n$ by~\reqref{deffulltwist}, hence $\alpha_{n-1}^n=1$ in $B_n(M)/\Gamma_2(P_n(M))$ using \repr{subs}(\ref{subs:5}). By \repr{action}, the action by conjugation of $\alpha_{n-1}$ on the elements of the basis $\{ a_{i,r} \,\mid\, \text{$i=1,\ldots,n$ and $r=1,\ldots, 2g$}\}$ of $P_n(M)/\Gamma_2(P_n(M))$ given by \reco{abel} is as follows:
\begin{equation}\label{eq:actalpha}
\alpha_{n-1}:\,
\text{$a_{1,r}\mapsto a_{2,r}\mapsto \cdots \mapsto a_{n-1,r}\mapsto a_{n,r}\mapsto a_{1,r}$ for all $r=1,\ldots, 2g$.}
\end{equation}
Using~\reqref{actalpha} and the fact that $\alpha_{n-1}^n=1$ in $B_n(M)/\Gamma_2(P_n(M))$, we have: 
\begin{equation}\label{eq:alphan}
(a_{1,1}\alpha_{n-1})^{n} = a_{1,1} a_{2,1} \cdots a_{n,1} ( \alpha_{n-1} )^n = a_{1,1} a_{2,1} \cdots a_{n,1} \Delta_n^2 = \prod_{i=1}^na_{i,1}.
\end{equation}
Let $X=\{ a_{1,1}\alpha_{n-1},\, a_{i,r}^n\,\vert\, \text{$1\leq i\leq n$ and $1\leq r\leq 2g$} \}$, let $Y=\{\, \prod_{i=1}^na_{i,1},\, a_{i,r}^n\,\vert\, \text{$1\leq i\leq n$ and $1\leq r\leq 2g$} \}$, and let $\widetilde{G}_{n,g}$ (resp.\ $L$) be the subgroup of $\sigma^{-1}(G_n)/\Gamma_2(P_n(M))$ generated by $X$ (resp.\ $Y$). Then the restriction $\map{\overline{\sigma}\bigl\lvert_{\widetilde{G}_{n,g}}\bigr.}{\widetilde{G}_{n,g}}[G_{n}]$ is surjective, and using~\reqref{alphan}, we see that $L$ is a subgroup of $\widetilde{G}_{n,g}$. We claim that $L=\kernb\bigl(\overline{\sigma}\bigl\lvert_{\widetilde{G}_{n,g}}\bigr.\bigr)$. Clearly, $L \subset\kernb\bigl(\overline{\sigma}\bigl\lvert_{\widetilde{G}_{n,g}}\bigr.\bigr)$. Conversely, let $w\in \kernb\bigl(\overline{\sigma}\bigl\lvert_{\widetilde{G}_{n,g}}\bigr.\bigr)$. Writing $w$ in terms of the generating set $X$ of $\widetilde{G}_{n,g}$ and using~\reqref{actalpha} and \reco{abel}, it follows that $\displaystyle w=(a_{1,1}\alpha_{n-1})^{m} \prod_{\substack{1\leq i\leq n\\ 1\leq r\leq 2g}} a_{i,r}^{n\delta_{i,r}}$, where $m\in \Z$ and $\delta_{i,r}\in \Z$ for all $1\leq i\leq n$ and $1\leq r\leq 2g$. The fact that $w\in \kernb\bigl(\overline{\sigma}\bigl\lvert_{\widetilde{G}_{n,g}}\bigr.\bigr)$ implies that $n$ divides $m$, and so $\displaystyle w=((a_{1,1}\alpha_{n-1})^{n})^{m/n} \prod_{\substack{1\leq i\leq n\\ 1\leq r\leq 2g}} a_{i,r}^{n\delta_{i,r}} \in L$, which proves the claim. Thus the following extension:
\begin{equation}\label{eq:sesGng}
1\to L \to \widetilde{G}_{n,g} \stackrel{\overline{\sigma}\bigl\lvert_{\widetilde{G}_{n,g}}\bigr.}{\to} G_{n} \to 1,
\end{equation}
is short exact. Now $L$ is also a subgroup of $P_{n}(M)/\Gamma_{2}(P_{n}(M))$, which is free Abelian of rank $2ng$ by \reco{abel}. Since $\{ a_{i,r} \,\mid\, \text{$i=1,\ldots,n$ and $r=1,\ldots, 2g$}\}$ is a basis of $P_{n}(M)/\Gamma_{2}(P_{n}(M))$, it follows from analysing $Y$ that $Y'=\{\, \prod_{i=1}^na_{i,1},\, a_{i,1}^n, a_{j,r}^n\,\vert\, \text{$2\leq i\leq n$, $1\leq j\leq n$ and $2\leq r\leq 2g$} \}$ is a basis of $L$. In particular $L$ is free Abelian of rank $2ng$. In terms of the basis $Y'$, the holonomy representation $\map{\rho}{G_n}[\operatorname{\text{Aut}}(L)]$ associated with the short exact sequence~\reqref{sesGng} is given by the block diagonal matrix:
\begin{equation}\label{eq:actmatrix}
\rho((1,n,n-1,\ldots,2))=\left( \begin{smallmatrix}
M_{1} & & &\\
& M_{2} & &\\
& & \ddots &\\
& & & M_{2g}
\end{smallmatrix}
\right),
\end{equation}
where $M_{1},\ldots,M_{2g}$ are the $n$-by-$n$ matrices satisfying:
\begin{equation*}
\text{$M_{1}=\left(
\begin{smallmatrix}
1 & 0 & 0 & \cdots & 0 & 0 & n\\
0 & 0 & 0 & \cdots & 0 & 0 & -1\\
0 & 1 & 0 & \cdots & 0 & 0 & -1\\
0 & 0 & 1 & \cdots & 0 & 0 & -1\\
\vdots & \vdots & \vdots & \ddots & \vdots & \vdots & \vdots\\
0 & 0 & 0 & \cdots & 1 & 0 & -1\\
0 & 0 & 0 & \cdots & 0 & 1 & -1
\end{smallmatrix}\right)$ and $M_{2}=\cdots=M_{2g}=
\left(
\begin{smallmatrix}
0 & 0 & 0 & \cdots & 0 & 0 & 1\\
1 & 0 & 0 & \cdots & 0 & 0 & 0\\
0 & 1 & 0 & \cdots & 0 & 0 & 0\\
0 & 0 & 1 & \cdots & 0 & 0 & 0\\
\vdots & \vdots & \vdots & \ddots & \vdots & \vdots & \vdots\\
0 & 0 & 0 & \cdots & 1 & 0 & 0\\
0 & 0 & 0 & \cdots & 0 & 1 & 0
\end{smallmatrix}
\right)$,}
\end{equation*}
where we have used the relation $a_{1,1}^n= (\prod_{i=1}^na_{i,1})^n\cdot \prod_{i=2}^n(a_{i,1}^n)^{-1}$. It follows from this description that $\rho$ is injective, and from \relem{cryst} and~\reqref{sesGng} that $\widetilde{G}_{n,g}$ is a crystallographic group of dimension $2ng$ and whose holonomy group is $\mathbb{Z}_n$. 

Now we prove that $\widetilde{G}_{n,g}$ is torsion free. Let $\omega\in \widetilde{G}_{n,g}$ be an element of finite order. Since $L$ is torsion free, the order of $\omega$ is equal to that of $\overline{\sigma}(\omega)$ in the cyclic group $G_{n}$, in particular $\omega^n=1$. Using~\reqref{actalpha} and~\reqref{alphan}, as well as the fact that $L$ is torsion free, there exist $\theta\in L$ and $j\in \{ 0,1,2,\ldots,n-1 \}$ such that $\omega=\theta(a_{1,1}\alpha_{n-1})^j$. Making use of the basis $Y'$ of $L$, 
\begin{equation}\label{eq:thetadev}
\theta=\biggl(\prod_{i=1}^na_{i,1}\biggr)^{\lambda_{1,1}}\ldotp \prod_{i=2}^na_{i,1}^{n\lambda_{i,1}} \ldotp \prod^{2g}_{r=2}\prod_{i=1}^na_{i,r}^{n\lambda_{i,r}},
\end{equation}
where $\lambda_{i,r}\in \mathbb{Z}$ for all $i=1,\ldots,n$ and $r=1,\ldots,2g$. On the other hand:
\begin{equation}\label{eq:omegan}
1=\omega^{n}= \biggl(\,\prod_{k=0}^{n-1} (a_{1,1}\alpha_{n-1})^{jk} \theta (a_{1,1}\alpha_{n-1})^{-jk}\biggr) (a_{1,1}\alpha_{n-1})^{nj}.
\end{equation}
By~\reqref{alphan}, $(a_{1,1}\alpha_{n-1})^{nj}=\prod_{i=1}^na_{i,1}^{j}$, and thus the right-hand side of~\reqref{omegan} belongs to $P_n(M)/\Gamma_2(P_n(M))$. We now compute the coefficient of $a_{1,1}$ in~\reqref{omegan} considered as one of the elements of the basis of $P_n(M)/\Gamma_2(P_n(M))$ given by \reco{abel}. From~\reqref{actalpha}, the terms appearing in the product $\prod^{2g}_{r=2}\prod_{i=1}^na_{i,r}^{n\lambda_{i,r}}$ of~\reqref{thetadev} do not contribute to this coefficient. Since conjugation by $a_{1,1}\alpha_{n-1}$ permutes cyclically the elements $a_{1,1},a_{2,1},\ldots,a_{n,1}$ by~\reqref{actalpha}, it follows that conjugation by $a_{1,1}\alpha_{n-1}$ leaves the first term $\prod_{i=1}^na_{i,1}^{\lambda_{1,1}}$ of~\reqref{thetadev} invariant, and so with respect to~\reqref{omegan}, it contributes $n\lambda_{1,1}$ to the coefficient of $a_{1,1}$. In a similar manner, with respect to~\reqref{omegan}, the second term of~\reqref{thetadev} contributes $n(\lambda_{2,1}+\cdots+\lambda_{n,1})$ to the coefficient of $a_{1,1}$. Putting together all of this information, the computation of the coefficient of $a_{1,1}$ in~\reqref{omegan} yields the relation $n(\lambda_{1,1}+\lambda_{2,1}+\cdots+\lambda_{n,1}) + j = 0$. We conclude that $j=0$, so $\omega=\theta=1$ because $L$ is torsion free, which using the first part of the statement, shows that $\widetilde{G}_{n,g}$ is a Bieberbach group. 

To prove the last part of the statement, using~\cite[Lemma~5.2(3)]{Sz} and the fact that $G_{n}$ is cyclic,
the centre $Z(\widetilde{G}_{n,g})$ of $\widetilde{G}_{n,g}$ is given by:
\begin{align}
Z(\widetilde{G}_{n,g}) &=\{ \theta\in L \,\vert\, \text{$\rho(g)(\theta)=\theta$ for all $g\in G_{n}$}\}= \{ \theta\in L \,\vert\, \text{$\rho((1,n,\ldots, 2))(\theta)=\theta$}\}.\label{eq:zGng}
\end{align}
To compute $Z(\widetilde{G}_{n,g})$, let $\theta\in L$. Writing $\theta$ with respect to the basis $Y'$ of $L$ as a vector $\left( \begin{smallmatrix}
\beta_{1}\\
\vdots\\
\beta_{2g}
\end{smallmatrix}
\right)$, where for all $i=1,\ldots,2g$, $\beta_{i}$ is a column vector with $n$ elements, and using the description of the action $\rho$ given by~\reqref{actmatrix}, it follows that $\theta\in Z(\widetilde{G}_{n,g})$ if and only if $\beta_{i}$ belongs to the eigenspace of $M_{i}$ with respect to the eigenvalue $1$ for all $i=1,\ldots,2g$. It is straightforward to see that these eigenspaces are of dimension $1$, and are generated by $\left( \begin{smallmatrix}
1\\
0\\
\vdots\\
0
\end{smallmatrix}
\right)$ if $i=1$ and by $\left( \begin{smallmatrix}
1\\
\vdots\\
1
\end{smallmatrix}
\right)$ if $i=2,\ldots,2g$. We conclude using~\reqref{zGng} that $Z(\widetilde{G}_{n,g})$ is the free Abelian group generated by $\{ \prod_{i=1}^na_{i,1}, \prod_{i=1}^na_{i,r}^n \,\vert\, 2\leq r\leq 2g\}$. This generating set may be seen to be a basis of $Z(\widetilde{G}_{n,g})$, in particular, $Z(\widetilde{G}_{n,g})$ is free Abelian of rank $2g$.
\end{proof}

We do not know whether $B_n(M)/\Gamma_2(P_n(M))$ admits a Bieberbach
subgroup of maximal rank whose holonomy group is non Abelian. The
following proposition shows that a certain Frobenius group cannot be the holonomy of any Bieberbach subgroup of $B_n(M)/\Gamma_2(P_n(M))$.

\begin{prop}\label{prop:frob}
Let $p$ be an odd prime, and let $M$ be a compact, orientable surface without boundary of genus $g\geq 1$. In $B_p(M)/\Gamma_{2}(P_p(M))$ there is no Bieberbach subgroup $H$ such that $\overline{\sigma}(H)$ is isomorphic to the Frobenius group $\mathbb{Z}_p\rtimes_{\theta}\mathbb{Z}_{(p-1)/2}$, where the automorphism $\theta(\iota_{(p-1)/2})$ is of order $(p-1)/2$, $\iota_{(p-1)/2}$ being a generator of $\mathbb{Z}_{(p-1)/2}$.
\end{prop}

\begin{proof}
Let $H$ be a subgroup of $B_p(M)/\Gamma_{2}(P_p(M))$ such that $\overline{\sigma}(H)$ is isomorphic to the Frobenius group $\mathbb{Z}_p\rtimes_{\theta}\mathbb{Z}_{(p-1)/2}$. Let us show that $H$ has non-trivial elements of finite order. Using \repr{toruscryst}, we also identify $B_p(M)/\Gamma_{2}(P_p(M))$ with $\Z^{2gp}\rtimes S_{p}$. Each element $B_p(M)/\Gamma_{2}(P_p(M))$ may thus be written as $x\ldotp \psi(w)$ where $x\in  P_p(M)/\Gamma_{2}(P_p(M))$ and $w\in S_p$, which we refer to as its normal form. As in the proofs of \reth{ordk0}(\ref{it:ordk0b}) and~\repr{embedd}, $P_p(M)/\Gamma_{2}(P_p(M))$ splits as a direct sum $\oplus_1^{2g}\mathbb{Z}^p$ that we interpret as a $\mathbb{Z}[S_p]$-module, the module structure being given by \repr{action}. If $z\in P_p(M)/\Gamma_{2}(P_p(M))$ then for $j=1,\ldots,2g$, let $z_j$ denote its projection onto the $j$th summand of $\oplus_1^{2g}\mathbb{Z}^p$, and 
for $(z,\tau)\in \mathbb{Z}^{2gp}\rtimes S_p$, let $(z,\tau)_j=(z_j,\tau)\in \mathbb{Z}^p\rtimes S_p$. Let $\map{\epsilon}{\mathbb{Z}^{p}}[\mathbb{Z}]$ denote the augmentation homomorphism. We extend $\epsilon$ to a map from $\mathbb{Z}^p\rtimes S_p$ to $\Z$, also denoted by $\epsilon$, by setting $\epsilon(z,\tau)=\epsilon(z)$ for all $(z,\tau)\in \mathbb{Z}^p\rtimes S_p$. Using the $\mathbb{Z}[S_p]$-module structure, observe that: 
\begin{equation}\label{eq:epsz}
\text{$\epsilon((\lambda z\lambda^{-1})_j)=\epsilon(z_j)$ for all $\lambda\in  B_p(M)/\Gamma_{2}(P_p(M))$ and $z\in P_p(M)/\Gamma_{2}(P_p(M))$.}
\end{equation}
Hence for all $(z,\tau),(z',\tau') \in \mathbb{Z}^p\rtimes S_p$:
\begin{equation*}
\epsilon(z\tau \ldotp z' \tau')= \epsilon(z\tau z' \tau^{-1}\ldotp \tau \tau')=  \epsilon(z\ldotp \tau  z' \tau^{-1})=\epsilon(z) \ldotp \epsilon(\tau  z' \tau^{-1})= \epsilon(z) +\epsilon(z')=\epsilon(z,\tau)+ \epsilon(z',\tau'),
\end{equation*}
and thus $\map{\epsilon}{\mathbb{Z}^p\rtimes S_p}[\Z]$ is a homomorphism. Identifying $\overline{\sigma}(H)$ with the Frobenius group $\mathbb{Z}_p\rtimes_{\theta}\mathbb{Z}_{(p-1)/2}$, let $w_1, w_2\in S_p$ be generators of $\mathbb{Z}_p$ and $\mathbb{Z}_{(p-1)/2}$ respectively. For $i=1,2$, let $v_{i}\in H$ be such that $\overline{\sigma}(v_{i})=w_{i}$, and let $a_{i}\in P_p(M)/\Gamma_{2}(P_p(M))$ be such that $v_{i}=a_{i} \psi(w_{i})$, where $\map{\psi}{S_p}[B_p(M)/\Gamma_2(P_p(M))]$ is the section for $\overline{\sigma}$ given in the proof of \repr{toruscryst}. Using the relation $w_2w_1w_2^{-1}=w_1^l$ in the Frobenius group, where $l$ is an element of the multiplicative group $\mathbb{Z}_p^{\ast}$ of order $(p-1)/2$, we have $v_2v_1v_2^{-1}v_1^{-l}\in  H\cap P_p(M)/\Gamma_{2}(P_p(M))$. Further:
\begin{align*}
v_2v_1v_2^{-1}v_1^{-l} =& a_{2} \psi(w_{2}) a_{1} \psi(w_{1}) \psi(w_{2})^{-1}a_{2}^{-1} (a_{1} \psi(w_{1}))^{-l}\\
=& a_{2} \psi(w_{2}) a_{1} \psi(w_{2})^{-1}a_{2}^{-1} \ldotp a_{2} \ldotp \psi(w_2w_1w_2^{-1}) a_{2}^{-1} \psi(w_2w_1w_2^{-1})^{-1}\ldotp \psi(w_2w_1w_2^{-1}w_1^{-l})\ldotp\\
& \prod_{k=1}^{l}(\psi(w_{1})^{l-k} a_{1}^{-1} \psi(w_{1})^{k-l}),
\end{align*}
and applying~\reqref{epsz} and using the relation $w_2w_1w_2^{-1}=w_1^l$, it follows that: 
\begin{equation}\label{eq:epsv1v2}
\epsilon((v_2v_1v_2^{-1}v_1^{-l})_i)= (1-l) \epsilon((a_{1})_{i})
\end{equation}
for all $1\leq i\leq 2g$. Let $v=v_2v_1v_2^{-1}v_1^{-l}v_1^{l-1}$. The element $v_1^{l-1}$ also belongs to $H$, so $v \in H$, and since $v_1^{l-1}=(a_{1} \psi(w_{1}))^{l-1}=\bigl(\prod_{k=0}^{l-2}(\psi(w_{1})^{k} a_{1} \psi(w_{1})^{-k})\bigr) \psi(w_{1})^{l-1}$, for all $1\leq i\leq 2g$, it follows that $(v)_{i}=\beta_{i} \psi(w_{1})^{l-1}$, where $\beta_{i}\in \Z^{p}$ is given by: 
\begin{equation}\label{eq:beta}
\beta_{i}= (v_2v_1v_2^{-1}v_1^{-l})_i \ldotp \biggl(\prod_{k=0}^{l-2}(\psi(w_{1})^{k} a_{1} \psi(w_{1})^{-k})\biggr)_{i}.
\end{equation}
Using~\reqref{epsz},~\reqref{epsv1v2} and~\reqref{beta}, we see that:
\begin{equation}\label{eq:epsbeta}
\epsilon(\beta_{i})= \epsilon((v_2v_1v_2^{-1}v_1^{-l})_i)+\epsilon \biggl(\biggl( \prod_{k=0}^{l-2}(\psi(w_{1})^{k} a_{1} \psi(w_{1})^{-k})\biggr)_{i}\biggr)=(1-l) \epsilon((a_{1})_{i})+ (l-1) \epsilon((a_{1})_{i})=0
\end{equation}
for all $1\leq i\leq 2g$. Now in normal form, $v$ may be written $v=(\beta_{1},\ldots, \beta_{2g}) \psi(w_{1})^{l-1}$. Since $w_1^{l-1}$ is non trivial, it follows that $v$ is non trivial. Taking $z=v$ and $k=p$ in \relem{coeftjr} and using~\reqref{epsbeta}  it follows that $v$ is of order $p$, and hence $H$ has non-trivial torsion elements. In particular, $H$ is not a Bieberbach group.
\end{proof}

It seems to be an interesting question to classify the subgroups of $S_n$ which can be the holonomy of a Bieberbach subgroup of $B_n(M)/\Gamma_2(P_n(M))$  of maximal rank. In the case where the subgroup is a semi-direct product, the argument given in the proof of \repr{frob} may be helpful in the study of  the problem.

Using the holonomy representation of the Bieberbach group $\widetilde{G}_{n,g}$ of \reth{bieberbach}, given in \req{actmatrix}, we now prove some dynamical and geometric properties of the flat manifold ${\mathcal X}_{n,g}$ whose fundamental group is $\widetilde{G}_{n,g}$. 

\begin{proof}[Proof of \reth{anosov}]
Let $n\geq 2$, let $g\geq 1$, let ${\mathcal X}_{n,g}$ be a Riemannian
compact flat manifold ${\mathcal X}_{n,g}$ whose fundamental group is $\widetilde{G}_{n,g}$, the Bieberbach group given in the statement of \reth{bieberbach}, and let $G_n$ be the cyclic group of that theorem. Let $1$ denote the generator $(1,n,n-1,\ldots,2)$ of $G_n$, and consider the holonomy representation $\map{\rho}{\mathbb{Z}_n} [\operatorname{\text{Aut}}(\mathbb{Z}^{2ng})]$ of $\widetilde{G}_{n,g}$ given in the proof of \reth{bieberbach}. By~\reqref{actmatrix}, if the characteristic polynomial of $\rho(1)$ is equal to $(x^n-1)^{2g}$, where $\map{\rho}{\mathbb{Z}_n} [\operatorname{\text{Aut}}(\mathbb{Z}^{2ng})]$ is the holonomy representation of $\widetilde{G}_{n,g}$. To see this, if $2\leq i\leq 2g$ then $M_{i}$ is the companion matrix of the polynomial $x^n-1$, and if we remove the first row and column of $M_{1}$, we obtain the companion matrix of the polynomial $1+x+x^2+\cdots +x^{n-1}$, so the characteristic polynomial of $M_{1}$ is also equal to $(x-1)(1+x+x^2+\cdots +x^{n-1})=x^n-1$. In particular, $\det(\rho(1))=1$, from which it follows from the end of~Section~\ref{sec:cryst} that ${\mathcal X}_{n,g}$ is orientable.
Further, the eigenvalues of $\rho(1)$ are the $n$th roots of unity each with multiplicity $2g$, and we conclude from \cite[Theorem~7.1]{Po} that ${\mathcal X}_{n,g}$ admits Anosov diffeomorphisms. By \cite[Theorems~6.4.12 and~6.4.13]{Dekimpe}, the first Betti number of $\mathcal{X}_{n,g}$ is given by:
\begin{align*}
 \beta_{1}({\mathcal X}_{n,g})&=2ng - \operatorname{\text{rank}}(\rho(1)- I_{2ng})= 2ng - 2g(n-1)= 2g. 
\end{align*}

It remains to show that ${\mathcal X}_{n,g}$ admits a K\"ahler structure. In order to do this, we make use of the following result from~\cite[Theorem~3.1 and Proposition~3.2]{JR} (see also \cite[Theorem~1.1 and Proposition~1.2]{DHS}) that a Bieberbach group $\Gamma$ of dimension $m$ is the fundamental group of a K\"ahler flat manifold with holonomy group $H$ if and only if $m$ is even, and each $\R$-irreducible summand of the holonomy representation $\map{\psi}{H}[\operatorname{\text{GL}}(m,\R)]$ of $\Gamma$, which is also $\mathbb{C}$-irreducible, occurs with even multiplicity. Since $\dim({\mathcal X}_{n,g})=2ng$ and the character vector of the representation $\rho$ is equal to $\left(\begin{smallmatrix}
2ng \\
0\\
\vdots\\
0
\end{smallmatrix}\right)$, it follows that each real irreducible representation of $\rho$ appears $2g$ times in its decomposition, and hence that ${\mathcal X}_{n,g}$ admits a K\"ahler structure.
\end{proof}

The Betti numbers of the K\"ahler manifold ${\mathcal X}_{n,g}$ may be computed using the formula $\beta_i({\mathcal X}_{n,g})=\textrm{dim}(\Lambda^i(\R^{2n}))^{G_{n}}$ of~\cite[Page~370]{DHS}. For real dimensions $4$ and $6$, we may identify the fundamental group of ${\mathcal X}_{n,g}$ in the CARAT classification~\cite{Carat} and~\cite[Tables, pp.~368 and~370]{DHS}. More precisely, the
CARAT symbol of the $4$-dimensional Bieberbach group $\widetilde{G}_{2,1}$ (resp.\ the $6$-dimensional Bieberbach group $\widetilde{G}_{3,1}$) is 18.1 (resp.\ 291.1).

\section{The cases of the sphere and non-orientable surfaces without boundary}\label{sec:4}

Let $M$ be a compact, connected surface without boundary. In this section, we describe the quotient group $B_n(M)/\Gamma_2(P_n(M))$ for the cases not covered by \repr{toruscryst}, namely $M$ is either the $2$-sphere $\St$ or a compact, non-orientable surface of genus $g\geq 1$ without boundary, which we denote by $N_{g}$. If $g=1$ then $N_{1}$ is the projective plane $\rp$. 

\begin{lem}\label{lem:torsion3}
Let  $M$ be either the $2$-sphere $\St$ or a compact, non-orientable surface of genus $g\geq 1$ without boundary, and let $n\in \N$. The group $P_n(M)/\Gamma_2(P_n(M))$ is isomorphic to: 
\begin{enumerate}[(a)]
\item\label{it:torsion3a} $\mathbb{Z}_2\oplus \mathbb{Z}^{n(n-3)/2}$ if $M=\St$ and $n\geq 3$. Further, the $\Gamma_2(P_n(M))$-coset of the full twist $\Delta_{n}^{2}$ is the generator of the summand $\mathbb{Z}_2$.
\item\label{it:torsion3b} $\mathbb{Z}_2^n \oplus\mathbb{Z}^{(g-1)n}$  if $M=N_g$.
\end{enumerate}
\end{lem}

\begin{proof} 
Part~(\ref{it:torsion3a}) follows from~\cite[page~674]{GGzeit}. For part~(\ref{it:torsion3b}), if $g=1$ then $M=\rp$, and the result is a consequence of~\cite[Proposition~8]{GG}. So suppose that $g\geq 2$. We make use of~\cite[Theorem~5.1, Presentation~3]{GM}. Since $T_{i,i}=1$ for all $1\leq i\leq n$~\cite[page~439]{GM} then by relation~(Pr3) of that presentation, $T_{i,j}=T_{i,j-1}$ in $P_n(M)/\Gamma_2(P_n(M))$ for all $1\leq i<j\leq n$, and it follows by induction on $j-i$ that $T_{i,j}=1$. We conclude from the presentation that $P_n(M)/\Gamma_2(P_n(M))$ is generated by $\brak{a_{j,r} \,\mid\, \text{$1\leq j\leq n$ and $1\leq r\leq g$}}$, subject to the relations $a^{2}_{j,1}a^{2}_{j,2}\cdots a^{2}_{j,g}=1$ for all $1\leq j\leq n$, from which we obtain the given isomorphism.
\end{proof}

\begin{prop}\label{prop:sdp1} 
Let $M=N_{g}$, where $g\geq 1$. In terms of the presentation of $B_{n}(M)$ given by~\cite[Theorem~2.2]{GM}, the map $\map{\psi}{S_n}[B_n(M)/\Gamma_2(P_n(M))]$ defined on the generating set $\{ \tau_{i} \,\vert\, i=1,\ldots, n-1\}$ of $S_{n}$ by $\psi(\tau_{i})=\sigma_i$ for all $1\leq i\leq n-1$ is an injective homomorphism. Consequently, the short exact sequence~\reqref{sestorus1} splits, and $B_n(M)/\Gamma_2(P_n(M)) \cong (\mathbb{Z}_2^n \oplus\mathbb{Z}^{n(g-1)}) \rtimes S_n$
\end{prop} 

\begin{proof}
Since $\tau_{1},\ldots,\tau_{n-1}$ and $\sigma_{1},\ldots,\sigma_{n-1}$ satisfy the Artin relations, and using~\reqref{sestorus1}, it suffices to show that $\sigma_i^2=1$ in $B_n(M)/\Gamma_2(P_n(M))$ for all $i=1,\ldots,n-1$. We now prove that this is the case. If $g=1$ (resp.\ $g\geq 2$), with the notation of~\cite[Theorem~7]{GG} (resp.~\cite[Theorem~2.2]{GM}), $\sigma_i^2=A_{i,i+1}$ (resp.\ $\sigma_i^2 =T_{i,i+1}$) in $P_{n}(M)$ for all $1\leq i\leq n-1$, so $\sigma_i^2$ belongs to $\Gamma_{2}(P_{n}(M))$ by~\cite[Proposition~8]{GG} (resp.\ by the proof of \relem{torsion3}), and thus its $\Gamma_{2}(P_{n}(M))$-coset is trivial in $B_n(M)/\Gamma_2(P_n(M))$ as required.
\end{proof}

By \relem{torsion3}, if $M=\St$ or $N_{g}$, where $g\geq 1$, $P_n(M)/\Gamma_2(P_n(M))$ has torsion, and we cannot use the methods of the proof of \repr{toruscryst} . But in fact, in these cases, the group $B_n(M)/\Gamma_2(P_n(M))$ is not crystallographic. To see this, we first prove the following lemma.
        
\begin{lem}\label{lem:crystnot}
Let $\Pi$ be a group, and suppose that there exists a group extension of the form:
\begin{equation*}
1\to T\times H \to \Pi \to \Phi \to 1,
\end{equation*}
where $H$ is torsion free, and $T$ is finite and non-trivial. Then $\Pi$ is not a crystallographic group.
\end{lem}

\begin{proof}
To prove that $\Pi$ is not crystallographic, by~\cite[page~34]{Dekimpe}, it suffices to show that it possesses a non-trivial, normal finite subgroup. Let us show that $T$ is such a subgroup. By the hypotheses, it suffices to prove that $T$ is normal. To see this, 
we view the kernel $T\times H$ as a subgroup of $\Pi$. Any inner automorphism of $\Pi$ thus restricts to an automorphism of $T \times H$. Since the image by such an automorphism of any element of $T$ is also a torsion element, it follows from the fact that $H$ is torsion free that $T$ is indeed normal in $\Pi$.
\end{proof}

\begin{prop}\label{prop:nonorientable}
Let $M=\St$ or $N_{g}$, where $g\geq 1$. Then for all $n\geq 1$, the quotient $B_n(M)/\Gamma_2(P_n(M))$ is not a crystallographic group.
\end{prop}

\begin{proof}
The result follows from~\reqref{sestorus1} and Lemmas~\ref{lem:torsion3} and~\ref{lem:crystnot}.
\end{proof}

\begin{rem}
If $M=\St$ (resp.\ $M=N_{g}$, where $g\geq 2$), the subgroup $T$ is that generated by the class of the full twist braid (resp.\ by $\left\{a_{j,1}a_{j,2}\cdots a_{j,g} \mid j=1,\ldots,n\right\}$ using the notation  of~\cite[Theorem~5.1]{GM}). If $M=\rp$ then $T=P_n(\rp)/\Gamma_2(P_n(\rp))$.
\end{rem}

\end{document}